\documentclass[11pt]{amsart}
\usepackage{amscd}
\usepackage{easybmat}
\usepackage{mathrsfs}
\usepackage{amsfonts}
\usepackage{color}
\usepackage{pifont}
\usepackage{upgreek}
\usepackage{bm}
\usepackage{shorttoc}

\usepackage{amsmath,amstext,amssymb,amscd}
\usepackage[colorlinks,linkcolor=blue,citecolor=blue, pdfstartview=FitH]{hyperref}
\usepackage[mathscr]{eucal}
\usepackage{mathrsfs}
\usepackage{epsf}
\numberwithin{equation}{section}

\def\p{\partial}
\def\o{\overline}
\def\b{\bar}
\def\mb{\mathbb}
\def\mc{\mathcal}
\def\n{\nabla}

\def\wt{\widetilde}

\def\l{\lrcorner}

\newtheorem{thm}{Theorem}[section]
\newtheorem{lemma}[thm]{Lemma}
\newtheorem{prop}[thm]{Proposition}
\newtheorem{cor}[thm]{Corollary}
\theoremstyle{definition}
\newtheorem{rem}[thm]{Remark}
\newtheorem{con}[thm]{Conjecture}

\theoremstyle{definition}
\newtheorem{defn}[thm]{Definition}
\newcommand{\comment}[1]{}

\usepackage{fancyhdr}
\begin{document}

\title{A logarithmic $\b{\p}$-equation on a compact K\"ahler manifold associated with a smooth divisor}

\author{Xueyuan Wan}
\thanks{Xueyuan Wan is partially supported by the National Natural Science Foundation of China (grant no. 12101093) and the Scientific Research Foundation of the Chongqing University of Technology. }
\address{Xueyuan Wan: Mathematical Science Research Center, Chongqing University of Technology, Chongqing 400054, China}
\email{xwan@cqut.edu.cn}

\begin{abstract}
 
 In this paper, we solve a logarithmic $\b{\p}$-equation on a compact K\"ahler manifold associated with a smooth divisor by using the cyclic covering trick. As applications, we discuss the closedness of logarithmic forms, injectivity theorems and obtain a kind of degeneration of spectral sequence at $E_1$. We also prove that the pair $(X,D)$ has unobstructed deformations for any smooth divisor $D\in|-2K_X|$. 
 \end{abstract}
\maketitle

\section*{Introduction}

It is well-known that Deligne's degeneration of logarithmic Hodge to de Rham spectral sequences at $E_1$-level \cite{Del71} is a fundamental result and has a significant impact in algebraic geometry for vanishing and injectivity theorems \cite{Viehweg, Fujino}. P. Deligne and L. Illusie \cite{Del87} also proved this degeneration by using a purely algebraic method based on positive characteristics. In terms of K\"ahler geometry, K. Liu, S. Rao and the author \cite{Wan} developed a method so that the degeneration of the spectral sequence can be reduced to solve a logarithmic $\b{\p}$-equation. By using the harmonic integral theory \cite{de, Kodaira}, we can solve the logarithmic $\b{\p}$-equation and thus give a geometric and simpler proof for Deligne's degeneration \cite[Theorem 0.4]{Wan}. In this paper, we will continue to study another kind of logarithmic $\b{\p}$-equations associated with a smooth divisor by combining with the cyclic covering trick. 

Let $X$ be a compact K\"ahler manifold of dimension $n$ and $D=\sum_{i=1}^rD_i$ be a smooth divisor in $X$, i.e., $D_i\cap D_j=\emptyset$ for $i\neq j$, each $D_i$ is smooth. Let $L$ be a holomorphic line bundle over $X$ with $L^N=\mc{O}_X(D)$ for some integer $N\geq 1$, and let $\Omega^p_X(\log D)$ denote the sheaf of differential forms with logarithmic poles along $D$. From Proposition \ref{prop1}, one can endow an integrable logarithmic connection $d=\p+\b{\p}$ along $D$ on $L^{-1}$ (see Definition \ref{defn1}), namely a $\mb{C}$-linear map
$$d: L^{-1}\to \Omega^1_X(\log D)\otimes L^{-1}$$
satisfying Leibniz's rule and $d^2=0$. Let $A^{0,q}(X, \Omega^p_X(\log D)\otimes L^{-1})$ denote the space of smooth $(0,q)$-forms with values in $ \Omega^p_X(\log D)\otimes L^{-1}$. For any $\alpha\in A^{0,q}(X, \Omega^p_X(\log D)\otimes L^{-1})$ with $\b{\p}\p\alpha=0$, we will consider the following
logarithmic $\b{\p}$-equation:
\begin{align}\label{0.1}
\b{\p}x=\p\alpha	
\end{align}
such that $x\in A^{0,q-1}(X,\Omega^{p+1}_X(\log D)\otimes L^{-1})$. In order to solve (\ref{0.1}), we will use the cyclic covering trick (see Subsection \ref{sub1.3}). More precisely, let $s$ be a canonical section of $\mc{O}_X(D)$ and taking a cyclic cover obtained by taking the $N$-th root out of $s$, we will get a compact K\"ahler manifold $Y$ and a finite morphism $\pi: Y\to X$. By pulling-back the logarithmic $\b{\p}$-equation (\ref{0.1}) to $Y$, one gets a solvable logarithmic $\b{\p}$-equation. This will lead to the vanishing of the harmonic projection of the logarthmic form (see (\ref{harmonic})). By these and the bundle-valued Hodge decomposition theorem, we have 
\begin{thm}[=Theorem \ref{thm1}]\label{thm0}
Let $X$ be a compact K\"ahler manifold and $D=\sum_{i=1}^rD_i$ be a smooth divisor in $X$, let $L$ be a holomorphic line bundle over $X$ with $L^N=\mc{O}_X(D)$.  
For any $\alpha\in A^{0,q}(X,\Omega^p_X(\log D)\otimes L^{-1})$ with $\b{\p}\p\alpha=0$, the following equation 
\begin{align*}
\b{\p}x=\p\alpha
\end{align*}
has a solution $x\in A^{0,q-1}(X,\Omega^{p+1}_X(\log D)\otimes L^{-1})$.	
\end{thm}

The $d$-closedness of logarithmic forms has been studied in \cite{Del71, Nog, Wan}. As an 
 immediate application of Theorem \ref{thm0}, we obtain the following result of the closedness of logarithmic forms. 
\begin{cor}[=Corollary \ref{Cor1}]\label{Cor0}
	If $\alpha\in A^{0,0}(X,\Omega^p_X(\log D)\otimes L^{-1})$ with $\b{\p}\p\alpha=0$, then $\p\alpha=0$.
\end{cor}

In \cite[Corollary 3.6]{Wan}, K. Liu, S. Rao and the author reproved an injectivity theorem of F. Ambro \cite[Theorem 2.1]{Ambro}, the key step is to prove the following mapping 
\begin{align*}
	\begin{CD}
		 H^q(X,\Omega^n_X(\log D))\to  H^{q+n}(X,\Omega^{\bullet}_X(\log D)),\quad [\alpha]\mapsto [\alpha]_d,
	\end{CD}
	\end{align*} is injective by solving a logarithmic $\b{\p}$-equation. As a comparison, 
the second application of Theorem \ref{thm0} is the following injectivity theorem. 
\begin{prop}[=Proposition \ref{Cor2}]\label{prop0}
	The following mapping is injective:
	\begin{align*}
	\begin{CD}
		&\iota: H^q(X,\Omega^n_X(\log D)\otimes L^{-1})\to H^{q+n}(X,\Omega^{\bullet}_X(\log D)\otimes L^{-1}),
	\end{CD}
	\end{align*}
where $\iota([\alpha])=[\alpha]_d$ and $H^{q+n}(X,\Omega^{\bullet}_X(\log D)\otimes L^{-1})$ denotes  the cohomology of the complex of sections $\left(\Gamma(X,\Omega^{\bullet}_X(\log D)\otimes L^{-1}),d\right)$ of $\Omega^{\bullet}_X(\log D)\otimes L^{-1}$.
\end{prop}

By logarithmic analogue of the general description on the terms in the Fr\"olicher spectral sequence as in \cite[Theorems 1 and 3]{cfgu}, $E_r^{p,q}\cong Z_r^{p,q}/B_r^{p,q}$ and $d_r: E_r^{p,q}\longrightarrow E_r^{p+r,q-r+1}$, where $Z_r^{p,q}$ and $B_r^{p,q}$  can be described as in (\ref{zr}) and (\ref{br}).  one would have $d_i=0, \forall i\geq 1$ if (\ref{0.1}) is solvable. Therefore, we get the following $E_1$-degeneration of spectral sequence. 
\begin{thm}[=Theorem \ref{thm2}]\label{thm0.1}
The spectral sequence 
\begin{align}
E^{p,q}_1=H^{q}(X,\Omega^p_X(\log D)\otimes L^{-1})\Longrightarrow \mb{H}^{p+q}(X,\Omega^{\bullet}_X(\log D)\otimes L^{-1})	
\end{align}
associated to the logarithmic de Rham complex 
$$(\Omega^{\bullet}_X(\log D)\otimes L^{-1},d)$$
degenerates in $E_1$. Here $\mb{H}^{p+q}(X,\Omega^{\bullet}_X(\log D)\otimes L^{-1})$ denotes the hypercohomology.
\end{thm}

It is well-known that a Calabi-Yau manifold has unobstructed deformations  \cite{T87, To89}. More precisely, if $X$ is a Calabi-Yau manifold, then for any $[\varphi_1]\in H^{1}(X,T_X)$, one can construct a holomorphic family $\varphi(t)\in A^{0,1}(X, T_X)$ on $t$, and it satisfies the  integrable equation $\b{\p}\varphi(t)=\frac{1}{2}[\varphi(t),\varphi(t)]$, $(\p\varphi(t)/\p t)|_{t=0}=\varphi_1$. In the case of logarithmic deformations, the set of infinitesimal logarithmic deformations is the space $H^1(X, T_X(-\log D))$. The pair $(X,D)$ has unobstructed  deformations if one can construct a holomorphic family
$$\varphi:=\varphi(t)\in A^{0,1}(X, T_X(-\log D))$$ satisfying the following integrability and initial conditions:
\begin{align*}
\b{\p}\varphi=\frac{1}{2}[\varphi,\varphi],\quad \frac{\p \varphi}{\p t}(0)=\varphi_1.	
\end{align*}
By using an iterative method originally from \cite{T87, To89, LSY} and developed in \cite{Liu, RZ, RZ2, RZ15, RwZ, RwZ1, Wan, Liu1}, we prove
\begin{thm}[=Theorem \ref{thm4}]\label{thm0.2}
Let $X$ be a compact K\"ahler manifold and $D$ a smooth divisor such that 
$D\in |-2K_X|$. Then, the pair $(X,D)$ has unobstructed deformations. 
\end{thm}
Note that for $X$ is projective, the above theorem was proved in \cite[Proposition 6.4, Remark 6.5]{Iacono} and \cite[Section 4.3.3]{KKP08} by a purely algebraic method. 

From the results \cite[Theorem 3.2 (b)]{Viehweg} and \cite[Section 4.3.3 (iii)]{KKP08}, there are more general $E_1$-degenerations of spectral sequences and the unobstructed deformations for smooth projective varieties. Both these results would be proved for compact K\"ahler manifolds if one can solve a more general logarithmic $\b{\p}$-equation. Based on these, it is natural to propose the following conjecture.
\begin{con}\label{con}
	Let $X$ be a compact K\"ahler manifold and $D=\sum_{i=1}^rD_i$ a simple normal crossing divisor in $X$. If $L$ is a holomorphic bundle over $X$ with $L^N=\mc{O}_X(\sum_{i=1}^r a_i D_i)$, $0\leq a_i\leq N$ and $a_i\in \mb{Z}$, then the following logarithmic $\b{\p}$-equation:
\begin{align}\label{dbar4}\b{\p}x=\p\alpha\end{align}
has a solution $x\in A^{0,q-1}(X,\Omega^{p+1}_X(\log D)\otimes L^{-1})$ for any $\alpha\in A^{0,q}(X,\Omega^p_X(\log D)\otimes L^{-1})$ with $\b{\p}\p\alpha=0$.
\end{con}

This article is organized as follows. In Section \ref{sec1}, we will recall some basic definitions and facts on logarithmic connection and cyclic covering, and prove that there is a canonical integrable logarithmic connection along $D$ on $L$.  In Section \ref{sec2}, by using the cyclic covering trick, we will solve a logarithmic $\b{\p}$-equation on a compact K\"ahler manifold associated to a smooth divisor, and then Theorem \ref{thm0} is proved in this section.  In Section \ref{sec3}, we will give some applications to Theorem \ref{thm0} and prove Corollary \ref{Cor0}, Proposition \ref{prop0}, Theorem \ref{thm0.1} and Theorem \ref{thm0.2}.  In Section \ref{sec5}, we make some further discussions and propose Conjecture \ref{con} on solving a more general logarithmic $\b{\p}$-equation.

\vspace{5mm}
\textbf{Acknowledgement}:
 The author would like to express his gratitude to Professor Kefeng Liu for suggesting
related problems and many insightful discussions on using the cyclic covering trick to study the geometry of logarithmic forms, and thank Professors  Sheng Rao and Xiaokui Yang for many helpful discussions.


\section{Preliminaries}\label{sec1}

In this section, we will recall some basic definitions and facts related to logarithmic connections and cyclic coverings. For more details, one may refer to \cite{Ancona, BHPV, EV06, Viehweg, Griffith, Kollar, Wan, Laz}.

\subsection{Logarithmic connection}\label{subsec1}

Let $X$ be a compact complex manifold of dimension $n$ and $D$ be a simple normal crossing divisor in $X$, i.e., $D=\sum_{i=1}^r D_i$, where  $D_i,1\leq i\leq r$ are distinct smooth hypersurfaces intersecting transversely in $X$.

Denote by $\tau: U:=X-D\to X$ the natural inclusion and
$$
\Omega^p_X(*D)=\lim_{\vec{\nu}}\Omega^p_X(\nu\cdot D)=\tau_*\Omega^p_U.
$$
Then $(\Omega^{\bullet}_X(*D),d)$ is a complex. The sheaf of logarithmic forms $$\Omega^p_X(\log D)$$ (introduced by Deligne in \cite{De}) is defined as the subsheaf of $\Omega^p_X(*D)$ with logarithmic poles along $D$, i.e., for any open subset $V\subset X$,
$$\Gamma(V,\Omega^p_X(\log D))=\{ \alpha\in \Gamma(V,\Omega^p_X(*D)): \alpha \,\,\text{and}\,\, d\alpha\,\, \text{have simple poles along}\,\, D\}.$$
 From (\cite[II, 3.1-3.7]{Del70} or \cite[Properties 2.2]{Viehweg}), the log complex $(\Omega^{\bullet}_X(\log D), d)$ is a subcomplex of $(\Omega^{\bullet}_X(*D), d)$ and $\Omega^p_X(\log D)$ is locally free,
 $$\Omega^p_X(\log D)=\wedge^p\Omega^1_X(\log D).$$

 For any $z\in X$,  we may choose local holomorphic coordinates $\{z^1,\cdots, z^n\}$ in a small neighborhood $U$ of $z=(0,\cdots, 0)$ such that
 $$D\cap U=\{z^1\cdots z^k=0\}$$
 is the union of coordinates hyperplanes. Such a pair
\begin{equation*}
(U,\{z^1,\cdots, z^n\})
\end{equation*}
is called a \textit{logarithmic coordinate system} \cite[Definition 1]{kawa}.
 Then $\Omega^p_X(\log D)$ is generated by the holomorphic forms and logarithmic differentials $dz^i/z^i$ ($i=1,\ldots, k$), i.e.,
 \begin{align*}
 \Omega^p_X(\log D)=\Omega^p_X\left\{\frac{dz^1}{z^1},\cdots,\frac{dz^k}{z^k}\right\}.	
 \end{align*}
Denote by $$A^{0,q}(X,\Omega^p_X(\log D))$$ the space of smooth $(0,q)$-forms on $X$ with values in $\Omega^p_X(\log D)$, and call an element of $A^{0,q}(X,\Omega^p_X(\log D))$  a \emph{logarithmic $(p,q)$-form}.

Now we recall the definition of logarithmic connections on a locally free coherent sheaf. 
 \begin{defn}[{\cite[Definition 2.4]{Viehweg}}]\label{defn1}
 	Let $\mc{E}$ be a locally free coherent sheaf on $X$ and let 
 	$$\n:\mc{E}\to \Omega^1_X(\log D)\otimes\mc{E}$$
 	be a $\mb{C}$-linear map satisfying
 	\begin{align}\label{1.11}\n(f\cdot e)=f\cdot\n e+df\otimes e.\end{align}
 	One defines 
 	$$\n_a:\Omega^a_X(\log D)\otimes \mc{E}\to \Omega^{a+1}_X(\log D)\otimes \mc{E}$$ 
 	by the rule 
 	$$\n_a(\omega\otimes e)=d\omega\otimes e+(-1)^a \omega\wedge \n e.$$
 	We assume that $\n_{a+1}\circ\n_a=0$ for all $a$. Such $\n$ will be called an {\it integrable logarithmic connection along $D$}, or just a connection. The complex 
 	$$(\Omega^{\bullet}_X(\log D)\otimes\mc{E}, \n_{\bullet})$$
 	is called the {\it logarithmic de Rham complex} of $(\mc{E}, \n)$. 
 \end{defn}

Let $L$ be a holomorphic line bundle over $X$ satisfying
\begin{align*}
L^N=\mc{O}_X\left(\sum_{i=1}^r a_i D_i\right)	
\end{align*}
for some $a_i\in\mb{Z}$, $1\leq i\leq r$. Then 
\begin{prop}\label{prop1}
There exists an integrable logarithmic connection along $D$ on $L$. 
\end{prop}
\begin{proof}
	Let $\sigma$ be the canonical meromorphic section of $\mc{O}_X\left(\sum_{i=1}^r a_i D_i\right)$, and let $e$ be a local frame of $L$ satisfying	
	\begin{align}\label{1.6}
		e^N=\prod_{i=1}^r (z^i)^{-a_i}\sigma.
	\end{align}
	We define 
	\begin{align}\label{1.5}
	\p e:=\frac{1}{N}\p\log \frac{e^N}{\sigma}  e=-\sum_{i=1}^{r}\frac{a_i}{N}\frac{dz^i}{z^i}e,	
	\end{align}
	which is taken valued in $\Omega^1_X(\log D)\otimes L$. It is well-defined since for any local frame $e'$ of $L$ with $e'^N=e^N$, one has  $\p(e'/e)=0$ so 
\begin{align*}
\p e'=-\sum_{i=1}^{r}\frac{a_i}{N}\frac{dz^i}{z^i}e'=\frac{\p e}{e} e'=\frac{e'}{e}\p e=\p(\frac{e'}{e})\otimes e+\frac{e'}{e}\p e,
\end{align*}
which satisfies (\ref{1.11}). Furthermore we define a logarithmic connection $d$ on $L$ by 
\begin{align*}
d:L\to \Omega^1_X(\log D)\otimes L\quad d (f\cdot e):=f\cdot\p e+d f\otimes e.
\end{align*}
 It induces a logarithmic connection on $\Omega^p_X(\log D)\otimes L$ by 
 \begin{align*}
 \begin{split}
 d(\omega\otimes e) =d\omega\otimes e+(-1)^{p}\omega\wedge \p e
 =d\omega\otimes e+(-1)^{p}\omega\wedge\left(-\sum_{i=1}^r\frac{a_i}{N}\frac{dz^i}{z^i}\right)\otimes e.
 \end{split}
 \end{align*}
By a direct calculation, one has $d^2=0$. Therefore, $d$ is an integrable logarithmic connection along $D$ on $L$.
\end{proof}

\subsection{Cyclic covering}\label{sub1.3}

In this subsection, let $L$ be a holomorphic line bundle over $X$ satisfies 
$$L^N=\mc{O}_X\left(\sum_{i=1}^r a_i D_i\right),$$
where $a_i\geq 0$ and $a_i\in\mb{Z}$. 
Let $s$ be the canonical section of $\mc{O}_X\left(\sum_{i=1}^r a_i D_i\right)$, and denote by $\mb{L}$ the total space of line bundle $L$. Let $\pi:\mb{L}\to X$ be the bundle projection. If $v\in \Gamma(\mb{L}, \pi^*L)$ is the tautological section, then the zero divisor of $\pi^*s-v^N$ defines an analytic subspace, say $X[\sqrt[n]{s}]$, in $\mb{L}$. We denote by $\o{X}[\sqrt[n]{s}]$ the normalization of $X[\sqrt[n]{s}]$, and $\o{\pi}:\o{X}[\sqrt[n]{s}]\to X$. We will call $\o{X}[\sqrt[n]{s}]$ the {\it cyclic cover obtained by taking the $N$-th root out of $s$}. The map $\o{\pi}$ is flat and finite morphism, $\o{X}[\sqrt[n]{s}]$ is a normal variety and 
\begin{align*}
\o{\pi}_*\mc{O}_{\o{X}[\sqrt[n]{s}]}=\bigoplus_{i=0}^{N-1}L^{-i}\left(\left\lfloor\frac{i}{N}\sum_{j=1}^ra_jD_j\right\rfloor\right),
\end{align*}
(see e.g. \cite[Section 2.11]{Kollar} or \cite[Corollary 3.11]{Viehweg}).

Let $e$ be  a local frame of $L$, any element of $L$ can be represented as the form $v \cdot e$, so as a complex manifold, the local coordinate  of $\mb{L}$ is given by $(z,v)$. In terms of local coordinates, then  
\begin{align*}
X[\sqrt[N]{s}]:=\{(z,v)\in \mb{L}|v^N-s(z)=0\},	
\end{align*}
where $s=s(z)e$. Therefore,  $(z,v)\in\mb{L}$ is a singular point of $X[\sqrt[n]{s}]$ if and only if 
\begin{align*}
\n(v^N-s(z))=(Nv^{N-1}, \n s(z))= 0,	
\end{align*}
which is equivalent to $z\in \text{Sing}(\sum_{i=1}^r a_i D_i)$. 

If $\sum_{i=1}^r a_i D_i$ is reduced and smooth, then $\text{Sing}(\sum_{i=1}^r a_i D_i)=\emptyset$ so 
 $X[\sqrt[n]{s}]$ is smooth. In this case, $X[\sqrt[n]{s}]=\o{X}[\sqrt[n]{s}]$. 
 \begin{prop}[{\cite[Lemma 17.1, 17.2]{BHPV}, \cite[Page 25]{Iacono}}]\label{prop2}
 Let $\pi: X[\sqrt[n]{s}]\to X$ be the $N$-cyclic covering of $X$ branched along a smooth divisor $D$ and determined by $L$, where $L^N=\mc{O}_X(D)$ and let $D'$ be the reduced divisor $\pi^{-1}(D)$ on $X[\sqrt[n]{s}]$, then 
\begin{itemize}
\item[(i)] $\mc{O}_{X[\sqrt[N]{s}]}(D')=\pi^*L$;
\item[(ii)]	$\pi^*D=N D'$;
\item[(iii)] $K_{X[\sqrt[N]{s}]}=\pi^*(K_X\otimes L^{N-1})$;
\item[(iv)] $\pi_*\mc{O}_X=\bigoplus_{i=0}^{N-1}L^{-i}$;
\item[(v)] $\pi^*\Omega^{\bullet}_X(\log D)=\Omega^{\bullet}_{X[\sqrt[n]{s}]}(\log (\pi^*D))$.
\end{itemize}
\end{prop}

\section{Logarithmic $\b{\p}$-equation}\label{sec2}

In this section, by using the cyclic covering trick, we will solve a logarithmic $\b{\p}$-equation on a compact K\"ahler manifold associated with a smooth divisor. 
Let $X$ be a compact K\"ahler manifold and  $D=\sum_{i=1}^r D_i$ be a smooth divisor, i.e. $D_i\cap D_j=\emptyset$ for $i\neq j$, each $D_i$ is smooth. Let $L$ be a holomorphic line bundle over $X$ satisfying $L^N=\mc{O}_X(D)$, and we consider the following logarithmic $\b{\p}$-equation:
\begin{align}\label{dbar1}
\b{\p}x=\p\alpha 	
\end{align}
for any $\alpha\in A^{0,q}(X,\Omega^p_X(\log D)\otimes L^{-1})$ with $\b{\p}\p\alpha=0$.

\subsection{Logarithmic $\b{\p}$-equation}

Denote by
\begin{align*}
	A^{0,q}(X,\Omega^p_X(\log D)\otimes L)
\end{align*}
the space of smooth $(0,q)$-forms with valued in $\Omega^p_X(\log D)\otimes L$. From Proposition \ref{prop1}, if $L^N=\mc{O}_X\left(\sum_{i=1}^r a_i D_i\right)$, then there is an integral logarithmic connection $d$ along $D$ on $L$. So there is a logarithmic de Rham complex of $(L, d)$, 
$$(\Omega^{\bullet}_X(\log D)\otimes L, d).$$
By taking $d=\p+\b{\p}$ and using Proposition \ref{prop1}, we have 
\begin{equation*}
  \p: A^{0,q}(X,\Omega^p_X(\log D)\otimes L)\to A^{0,q}(X,\Omega^{p+1}_X(\log D)\otimes L)
\end{equation*}
and 
\begin{equation*}
  \b{\p}: A^{0,q}(X,\Omega^p_X(\log D)\otimes L)\to A^{0,q+1}(X,\Omega^{p}_X(\log D)\otimes L).
\end{equation*}
For any $\alpha\in A^{0,q}(X,\Omega^p_X(\log D)\otimes L)$ with $\b{\p}\p\alpha=0$, one may ask that whether the following logarithmic $\b{\p}$-equation
\begin{align}\label{dbar}
\b{\p}x=\p\alpha 	
\end{align}
has a solution $x\in A^{0,q-1}(X,\Omega^{p+1}_X(\log D)\otimes L)$. 
\begin{rem}
For the case that $L$ is trivial or $L=\mc{O}_X(-D)$, and $X$ is a compact K\"ahler manifold, the logarithmic $\b{\p}$-equation (\ref{dbar}) is solvable \cite[Theorem 0.1, Theorem 0.2]{Wan}. 
\end{rem}

\subsection{Solving the $\b{\p}$-equation} In this subsection, we will solve the logarithmic $\b{\p}$-equation (\ref{dbar1}) by using the cyclic covering trick. 

Since $D=\sum_{i=1}^r D_i$ is a smooth divisor in $X$ and $L^N=\mc{O}_X(D)$, from the discussions in Subsection \ref{sub1.3}, $Y:=X[\sqrt[N]{s}]$ is a compact complex submanifold of $\mb{L}$. Let $\pi: Y\to X$ be the cyclic covering. For any point $x\in D$, one can take a small neighborhood $U$ around $x$, such that 
\begin{align*}
U\cap D=\{z^1\cdots z^k=0\}.	
\end{align*}
Since $D$ is smooth, by take $U$ small sufficiently, so $k=1$. 
 Without loss of generality,  for any $x\in D$,  one may take a local coordinate neighborhood around $x$ such that 
\begin{align*}
\pi: (w^1,\cdots,w^n)\mapsto (z^1=(w^1)^{N}, z^2=w^2,\cdots, z^n=w^n).	
\end{align*}
For any $\alpha\in A^{0,q}(X,\Omega^p_X(\log D)\otimes L^{-1})$, locally 
\begin{align*}
\alpha=f \frac{dz^1}{z^1}\wedge dz^2\wedge \cdots\wedge dz^p\otimes e^*	
\end{align*}
where $f$ is a local smooth $(0,q)$-form and $e^*$ is a local frame of $L^{-1}$, then 
\begin{align}\label{1.2}
\begin{split}
\pi^*\alpha=\pi^*f N\frac{dw^1}{w^1}\wedge dw^2\wedge \cdots\wedge dw^p\otimes \pi^*e^* &\in A^{0,q}(Y,\Omega^p_Y(\log (\pi^*D))\otimes \pi^*L^{-1})\\
&\cong A^{0,q}(Y,\Omega^p_Y(\log D')\otimes \mc{O}_Y(-D')),
\end{split}	
\end{align}
where the isomorphism  follows from Proposition \ref{prop2} (i), which is given by dividing the canonical section  of $\mc{O}_X(-D')$.

Firstly, we note that 
\begin{align}\label{1.3}
	 A^{0,q}(Y,\Omega^p_Y(\log D')\otimes \mc{O}_Y(-D'))\cong  A^{n,q}(Y, T^{n-p}_Y(-\log D'))\hookrightarrow A^{p,q}(Y).
\end{align}
where $T^{n-p}_Y(-\log D'):=\wedge^{n-p}\left(\Omega^1_Y(\log D')\right)^*$.
In fact, the first isomorphism is given by the following: for any $\beta\in A^{0,q}(Y,\Omega^p_Y(\log D')\otimes \mc{O}_Y(-D'))$, in terms of local coordinates, one has
\begin{align*}
\begin{split}
\beta 
&=g \frac{dw^1}{w^1}\wedge dw^2\wedge \cdots\wedge dw^p w^1\\
&=gdw^1\wedge \cdots\wedge dw^p\\
&\simeq g	dw^1\wedge \cdots\wedge dw^n\otimes \left(\frac{\p}{\p w^{p+1}}\otimes \cdots\otimes \frac{\p}{\p w^r}\otimes\frac{\p}{\p w^{r+1}}\otimes\cdots\otimes \frac{\p}{\p w^n}\right)\\
&\in A^{n,q}(Y, T^{n-p}_Y(-\log D')).
\end{split}
\end{align*}
Here $g$ is a local smooth $(0,q)$-form. 

Let $\sigma'$ be the canonical meromorphic section of $\mc{O}_Y(-D')$, then 
\begin{align}\label{1.1}
 \begin{split}
 	\p(\pi^* e^*)&:=\p\log \frac{\pi^* e^*}{\sigma'}\otimes \pi^*e^*\\
 	&=\frac{1}{N}\p\log \frac{\pi^*e^{*N}}{\sigma'^N}\otimes \pi^*e^*\\
 	&=\pi^*(\p e^*), 
 \end{split}	
 \end{align}
where the last equality holds by (\ref{1.5}) and Proposition \ref{prop2} (ii). 

 For any $\alpha\in A^{0,q}(X,\Omega^{p-1}_X(\log D)\otimes L^{-1})$ with $\b{\p}\p\alpha=0$, by (\ref{1.1}), then 
 \begin{align*}
 \b{\p}\p((\sigma')^{-1}\pi^*\alpha)=\b{\p}((\sigma')^{-1}\p(\pi^*\alpha))=(\sigma')^{-1}\pi^*(\b{\p}\p\alpha)=0.	
 \end{align*}
From (\ref{1.2}) and (\ref{1.3}), we have 
\begin{align*}
(\sigma')^{-1}\pi^*\alpha\in 	A^{n,q}(Y, T^{n-p}_Y(-\log D'))\hookrightarrow A^{p,q}(Y).
\end{align*}
\begin{prop}\label{prop3}
If 	$(X,\omega)$ is a compact K\"ahler manifold, then $Y$ is also a compact K\"ahler manifold. 
\end{prop}
\begin{proof}
By the construction of cyclic covering, $Y$ is a compact complex submanifold of the total space $\mb{L}$. Let $h$ be a Hermitian metric on $L$, since $Y$ is compact, we may assume that 
\begin{align}\label{2.1}
Y\hookrightarrow Y_R:=\{v e\in L, \|v\|^2_h=h|v|^2<R\}\subset \mb{L},	
\end{align}
for some large $R>0$. By a direct calculation, one has 
\begin{align*}
\sqrt{-1}\p\b{\p}\|v\|^2_h=\sqrt{-1}(\p\b{\p}\log h)\|v\|^2_h+\sqrt{-1}h\delta v\wedge \delta \b{v},
\end{align*}
where $\delta v:=dv+v\p\log h$. When restricting to the open manifold $Y_R$, one has 
\begin{align*}
\sqrt{-1}\p\b{\p}\|v\|^2_h\geq  -C R\pi^*\omega+\sqrt{-1}h\delta v\wedge \delta \b{v},	
\end{align*}
where $C>0$ is a uniform constant satisfying $\sqrt{-1}\p\b{\p}\log h>-C\omega$, and $\pi: Y_R(\subset\mb{L})\to X$ is the natural projection.
Thus, $C_1\pi^*\omega+\sqrt{-1}\p\b{\p}\|v\|^2_h$ is a K\"ahler metric on $Y_R$ for any $C_1>CR$. From (\ref{2.1}) and by the restriction, one obtains a K\"ahler metric on $Y$. 	
\end{proof}
From \cite[Theorem 0.2]{Wan}, one can solve the equation $\b{\p}x=\p((\sigma')^{-1}\pi^*\alpha)$ with 
\begin{align}\label{1.4}
x=\mc{I}^*\b{\p}^*_E\mb{G}''_E(\mc{I}^*)^{-1}\p((\sigma')^{-1}\pi^*\alpha)\in A^{n,q-1}(Y, T^{n-p-1}_Y(-\log D')).
\end{align}
Here $(\mc{I}^*)^{-1}: A^{n,q}(Y, T^{n-p-1}_Y(-\log D')))\to A^{n,q}(Y, (E_Y^{n-p-1})^*)$ is the canonical isomorphism, $E_Y^{n-p-1}$ is the holomorphic vector bundle associated to the local free sheaf $\Omega^{n-p-1}_Y(\log D')$. 

Now we define the following isomorphism by 
\begin{align*}
\begin{split}
(\mc{I}^*)^{-1}:& A^{0,q}(X,\Omega^{p+1}_X(\log D)\otimes L^{-1})\to A^{n,q}(X, (E_X^{n-p-1})^*\otimes L^{N-1})\\
&\alpha=f\frac{dz^1}{z^1}\wedge dz^2\wedge\cdots\wedge dz^{p+1}\otimes e^*\mapsto fdz\otimes (e_{p+2}\otimes \cdots\otimes e_n )\otimes (e^*)^{1-N},
\end{split}
\end{align*}
where $\{e_{i_1}\otimes\cdots\otimes e_{i_{n-p-1}}\}$ is a basis of $(E^{n-p-1}_X)^*$.  The definition is well-defined since 
\begin{align*}
	\alpha&=f\frac{dz^1}{z^1}\wedge dz^2\wedge\cdots\wedge dz^{p+1}\otimes e^*\\
	&\simeq f\frac{dz}{z^1}\otimes (\frac{\p}{z^{p+2}}\otimes\cdots\otimes\frac{\p}{\p z^n})\otimes e^*\\
	&=f\frac{dz}{e^N/\sigma}\otimes (\frac{\p}{z^{p+2}}\otimes\cdots\otimes\frac{\p}{\p z^n})\otimes e^*\\
	&=f dz\otimes (\frac{\p}{z^{p+2}}\otimes\cdots\otimes\frac{\p}{\p z^n})\otimes (e^*)^{1-N}\otimes \sigma\\
	&\simeq fdz\otimes (e_{p+2}\otimes \cdots\otimes e_n )\otimes (e^*)^{1-N}\otimes \sigma\\
	&\simeq fdz\otimes (e_{p+2}\otimes \cdots\otimes e_n )\otimes (e^*)^{1-N}.
\end{align*}
Here we denote $dz:=dz^1\wedge \cdots\wedge dz^n$. 
So one has the following lemma.
\begin{lemma}\label{lemma1}
It holds the following commutative diagram:
\begin{equation*}
	\begin{CD}
A^{n,q}(Y, T^{n-p-1}_Y(-\log D')) @>(\mc{I}^*)^{-1}>> A^{n,q}(Y, (E_Y^{n-p-1})^*)\\
@AA(\sigma')^{-1}\cdot\pi^*A @AA(\sigma')^{N-1}\cdot\pi^*A\\
A^{0,q}(X,\Omega^{p+1}_X(\log D)\otimes L^{-1}) @>(\mc{I}^*)^{-1}>> A^{n,q}(X, (E_X^{n-p-1})^*\otimes L^{N-1})
\end{CD}
\end{equation*}
\end{lemma}
\begin{proof}
One can check the above  commutative diagram directly. For any 	$$\alpha=f\frac{dz^1}{z^1}\wedge dz^2\wedge\cdots\wedge dz^{p+1}\otimes e^*\in A^{0,q}(X,\Omega^{p+1}_X(\log D)\otimes L^{-1}),$$ one has 
\begin{align*}
\begin{split}
(\sigma')^{N-1}\pi^*((\mc{I}^*)^{-1}(\alpha))&=(\sigma')^{N-1}\pi^*(fdz\otimes (e_{p+2}\otimes \cdots\otimes e_n )\otimes (e^*)^{1-N})\\
&=	N \pi^*f \wedge dw\otimes(\wt{e}_{p+2}\otimes\cdots\otimes\wt{e}_{n})(\pi^*e^*/(w_1\sigma'))^{1-N}\\
&=N \pi^*f \wedge dw\otimes(\wt{e}_{p+2}\otimes\cdots\otimes\wt{e}_{n})(\pi^*e^*/(w_1\sigma')),
\end{split}
\end{align*}
where $\wt{e}_i:=\pi^* e_i$, $dw=dw^1\wedge\cdots\wedge dw^n$  and the third equality holds since one can take the local frame $e^*$ satisfies $(\pi^*e^*/(w_1\sigma'))^{N}=1$ as (\ref{1.6}). While
\begin{align*}
\begin{split}
	(\mc{I}^*)^{-1}(\sigma')^{-1}\pi^*(\alpha)&=(\mc{I}^*)^{-1}(\sigma')^{-1}\left(N\pi^* f\frac{dw^1}{w^1}\wedge dw^2\wedge \cdots\wedge dw^p\otimes \pi^*e^*\right)\\
	&=N \pi^*f \wedge dw\otimes(\wt{e}_{p+2}\otimes\cdots\otimes\wt{e}_{n})(\pi^*e^*/(w_1\sigma'))\\
	&=(\sigma')^{N-1}\pi^*((\mc{I}^*)^{-1}(\alpha)).
\end{split}	
\end{align*}

\end{proof}
From (\ref{1.4}) and Lemma \ref{lemma1}, one has 
$$(\mc{I}^*)^{-1}x=\b{\p}^*_E\mb{G}''_E(\mc{I}^*)^{-1}\p((\sigma')^{-1}\pi^*\alpha),$$
which is a solution of $\b{\p}y= (\mc{I}^*)^{-1}(\sigma')^{-1}\pi^*(\p\alpha)=(\sigma')^{N-1}\pi^*((\mc{I}^*)^{-1}\p\alpha)=(\sigma')^{N-1}\pi^*\beta$, where $\beta=(\mc{I}^*)^{-1}\p\alpha\in A^{n,q}(X, (E_X^{n-p-1})^*\otimes L^{N-1})$. For any smooth Hermitian metric $h$ on the vector space $(E_X^{n-p-1})^*\otimes L^{N-1}$ and any harmonic element  $\eta\in A^{n,q}(X, (E_X^{n-p-1})^*\otimes L^{N-1})$, i.e. $\b{\p}\eta=0=\b{\p}*'\eta$, where $*'$ is the Hodge $*$-operator. Then 
\begin{align}
\begin{split}
\langle\beta,\eta\rangle &=\int_X \beta\wedge *'\eta\\
&=\frac{1}{N}\int_Y \pi^*(\beta\wedge *'\eta)	\\
&=\frac{1}{N}\int_Y (\sigma')^{1-N}\b{\p}((\mc{I}^*)^{-1}x)\wedge \pi^*(*'\eta)\\
&=\frac{1}{N}(-1)^{n+q}\int_Y(\sigma')^{1-N}(\mc{I}^*)^{-1}x\wedge \pi^*\b{\p}(*'\eta)=0,
\end{split}
\end{align}
where the second equality follows from \cite[Lemma 2.2]{Wells}. It follows that 
\begin{align}\label{harmonic}
	\mb{H}(\beta)=0.
\end{align}
 By the Hodge theorem for bundle-valued and noting $\b{\p}\beta=0$, we have 
\begin{align*}
\beta=\b{\p}\b{\p}^*\mb{G}\beta+\b{\p}^*\b{\p}\mb{G}\beta+\mb{H}(\beta)=\b{\p}\b{\p}^*\mb{G}\beta=\b{\p}\b{\p}^*\mb{G}(\mc{I}^*)^{-1}\p\alpha, 	
\end{align*}
which is equivalent to 
\begin{align*}
(\mc{I}^*)^{-1}\left(\b{\p}(\mc{I}^*\b{\p}^*\mb{G}(\mc{I}^*)^{-1}\p\alpha)-\p\alpha\right)=0.	
\end{align*}
Since $(\mc{I}^*)^{-1}$ is an isomorphism, so we obtain a solution \begin{align}\label{2.2}
 x=\mc{I}^*\b{\p}^*\mb{G}(\mc{I}^*)^{-1}\p\alpha\in A^{0,q-1}(X,\Omega^{p+1}_X(\log D)\otimes L^{-1})	
 \end{align}
of the equation $\b{\p}x=\p\alpha$.  In one word, we obtain
\begin{thm}\label{thm1}
Let $X$ be a compact K\"ahler manifold and $D=\sum_{i=1}^rD_i$ be a smooth divisor in $X$, let $L$ be a holomorphic line bundle over $X$ with $L^N=\mc{O}_X(D)$.  
For any $\alpha\in A^{0,q}(X,\Omega^p_X(\log D)\otimes L^{-1})$ with $\b{\p}\p\alpha=0$, the following equation 
\begin{align}\label{dbar3}
\b{\p}x=\p\alpha
\end{align}
has a solution $x\in A^{0,q-1}(X,\Omega^{p+1}_X(\log D)\otimes L^{-1})$.	
\end{thm}

\section{Some applications}\label{sec3}

In this section, we will give some applications to Theorem \ref{thm1}. Throughout this section, let $X$ be a compact K\"ahler manifold of dimension $n$ and $D=\sum_{i=1}^r D_i$ be a smooth divisor, let $L$ be a holomorphic line bundle over $X$ with $L^N=\mc{O}_X(D)$. 

\subsection{Closedness of logarithmic forms}

The theory of logarithmic forms has played a very important role in various aspects of analytic-algebraic geometry, in which the understanding of the closedness of logarithmic forms is fundamental. In $1971$, Deligne \cite[(3.2.14)]{Del71} proved the $d$-closedness of logarithmic forms on a smooth complex quasi-projective variety by showing the degeneration of logarithmic Hodge to de Rham spectral sequence.  In $1995$,  Noguchi \cite{Nog} gave a short proof of this result. In \cite[Corollary 0.3]{Wan}, K. Liu, S. Rao and the author generalized Deligne's result and obtained: if $\alpha\in A^{0,0}(X,\Omega^p_X(\log D))$ with $\b{\p}\p\alpha=0$ then $\p\alpha=0$. 

As the first application of Theorem $\ref{thm1}$, we obtain
\begin{cor}\label{Cor1}
	If $\alpha\in A^{0,0}(X,\Omega^p_X(\log D)\otimes L^{-1})$ with $\b{\p}\p\alpha=0$, then $\p\alpha=0$.
\end{cor}
\begin{proof}
	By Theorem \ref{thm1}, there exists a solution $$x\in A^{0,-1}(X,\Omega^p_X(\log D)\otimes L^{-1})=\{0\}$$ such that $\b{\p}x=\p\alpha$ and then
	$$
	\p\alpha=\b{\p}x=0
	$$
since $x=0$.
\end{proof}

\subsection{An injectivity theorem}

In this subsection, we will prove an injectivity theorem by using Theorem \ref{thm1}. 

From Proposition \ref{prop1}, $d$ is an integrable logarithmic connection along $D$ on $L^{-1}$, and 
$$(\Omega^{\bullet}_X(\log D)\otimes L^{-1}, d)$$
is a logarithmic de Rham complex of $(L^{-1}, d)$. Let 
$$H^{k}(X,\Omega^{\bullet}_X(\log D)\otimes L^{-1})$$ denote the cohomology of the complex of sections $\left(\Gamma(X,\Omega^{\bullet}_X(\log D)\otimes L^{-1}),d\right)$ of $\Omega^{\bullet}_X(\log D)\otimes L^{-1}$. 
\begin{prop}\label{Cor2}
	The following mapping is injective:
	\begin{align*}
	\begin{CD}
		\iota: H^q(X,\Omega^n_X(\log D)\otimes L^{-1})\to H^{q+n}(X,\Omega^{\bullet}_X(\log D)\otimes L^{-1}),\quad [\alpha]\mapsto \iota([\alpha]):=[\alpha]_d.
	\end{CD}
	\end{align*}
\end{prop}
\begin{proof}
For any $[\alpha]\in H^q(X,\Omega^n_X(\log D)\otimes L^{-1})$, then $\b{\p}\alpha=0$. By considering the degree of $\alpha$, so 
$$d\alpha=\p\alpha+\b{\p}\alpha=\b{\p}\alpha=0.$$
	It follows that $\iota([\alpha])=[\alpha]_d\in  H^{q+n}(X,\Omega^{\bullet}_X(\log D)\otimes L^{-1})$. If $\iota([\alpha])=[\alpha]_d=0$, then there exists a logarithmic form $\beta\in A^{q+n-1}(X, \Omega^{\bullet}_X(\log D)\otimes L^{-1})$ such that $\alpha=d\beta$. Therefore, the components $\beta_{n-1,q}\in A^{0,q}(X,\Omega^{n-1}_X(\log D)\otimes L^{-1})$ and $\beta_{n,q-1}\in A^{0,q-1}(X,\Omega^{n}_X(\log D)\otimes L^{-1})$ of $\beta$ satisfying 
	\begin{align}\label{2.3}
	\alpha=\p\beta_{n-1,q}+\b{\p}\beta_{n,q-1}.	
	\end{align}
Notice that $\b{\p}\beta_{n-1,q}=\b{\p}\alpha=0$, by Theorem \ref{thm1}, there exists $\gamma\in A^{0,q-1}(X,\Omega^{n}_X(\log D)\otimes L^{-1})$ such that
\begin{align}\label{2.4}
\p\beta_{n-1,q}=\b{\p}\gamma. 	
\end{align}
Combining (\ref{2.3}) with (\ref{2.4}), one has 
$$\alpha=\b{\p}(\gamma+\beta_{n,q-1}),$$
which implies that $[\alpha]=0\in H^q(X,\Omega^n_X(\log D)\otimes L^{-1})$. So we get the injectivity of $\iota$. 
\end{proof}
\begin{rem}
For the case that $L$ is trivial and $D=\sum_{i=1}^rD_i$ is a simple normal crossing divisor, Proposition \ref{Cor2} was proved in \cite[Corollary 3.6]{Wan}. More precisely, we proved that the restriction homomorphism
$$H^q(X,\Omega^n_X(\log D))\to H^q(X-D,K_{X-D})$$
is injective, which was first proved in \cite[Theorem 2.1]{Ambro} by using algebraic method. 
\end{rem}

\subsection{Degeneration of spectral sequences}

In this subsection, as an application of Theorem \ref{thm1}, we will prove the following $E_1$-degeneration of spectral sequences. 
\begin{thm}\label{thm2}
The spectral sequence 
\begin{align}
E^{p,q}_1=H^{q}(X,\Omega^p_X(\log D)\otimes L^{-1})\Longrightarrow \mb{H}^{p+q}(X,\Omega^{\bullet}_X(\log D)\otimes L^{-1})	
\end{align}
associated to the logarithmic de Rham complex 
$$(\Omega^{\bullet}_X(\log D)\otimes L^{-1},d)$$
degenerates in $E_1$. Here $\mb{H}^{p+q}(X,\Omega^{\bullet}_X(\log D)\otimes L^{-1})$ denotes the hypercohomology.
\end{thm}
 \begin{proof}
The proof needs a logarithmic analogue of the general description on the terms in the Fr\"olicher spectral sequence as in \cite[Theorems 1 and 3]{cfgu}.
By Dolbeault isomorphism theorem, one has
$$
H^q(X,\Omega^p_X(\log D)\otimes L^{-1})\cong H^{0,q}_{\b{\p}}(X,\Omega^p_X(\log D)\otimes L^{-1}):=\frac{\text{Ker}(\b{\p})\cap A^{0,q}(X,\Omega^p_X(\log D)\otimes L^{-1})}{\b{\p}A^{0,q-1}(X,\Omega^p_X(\log D)\otimes L^{-1})}.
$$
By the definition of spectral sequences
 $$E_r^{p,q}\cong Z_r^{p,q}/B_r^{p,q},$$
where $Z_r^{p,q}$ lies between the $\b{\p}$-closed and $d$-closed logarithmic $(p,q)$-forms and $B_r^{p,q}$ lies between the $\b{\p}$-exact and $d$-exact logarithmic $(p,q)$-forms
in some senses.
Actually,
$$Z_1^{p,q}=\{\alpha\in A^{0,q}(X,\Omega^p_X(\log D)\otimes L^{-1})\ |\ \b{\p}\alpha=0\},$$
$$B_1^{p,q}=\{\alpha\in A^{0,q}(X,\Omega^p_X(\log D)\otimes L^{-1})\ |\ \alpha=\b{\p}\beta, \beta\in A^{0,q-1}(X,\Omega^p_X(\log D)\otimes L^{-1})\}.$$
For $r\geq 2$,
\begin{equation}\label{zr}
\begin{aligned}
Z_r^{p,q}=\{\alpha_{p,q}\in A^{0,q}(X,\Omega^p_X(\log D)\otimes L^{-1})\ |
 &\ \b{\p}\alpha_{p,q}=0,\ \text{and there exist }\\
 &\text{$\alpha_{p+i,q-i}\in A^{0,q-i}(X,\Omega^{p+i}_X(\log D)\otimes L^{-1})$ }\\
  &\text{such that $\p\alpha_{p+i-1,q-i+1}+\b\p\alpha_{p+i,q-i}=0, 1\leq i\leq r-1$}
 \},
\end{aligned}
\end{equation}
\begin{align}\label{br}
\begin{split}
B_r^{p,q}=\{&\p\beta_{p-1,q}+\b\p\beta_{p,q-1}\in A^{0,q}(X,\Omega^p_X(\log D)\otimes L^{-1})\ |
 \ \text{there exist }\\
 &\text{ $\beta_{p-i,q+i-1}\in A^{0,q+i-1}(X,\Omega^{p-i}_X(\log D)\otimes L^{-1})$, $2\leq i\leq r-1, $}\\
 &\text{ such that $\p\beta_{p-i,q+i-1}+\b\p\beta_{p-i+1,q+i-2}=0, \b\p\beta_{p-r+1,q+r-2}=0$}\},
 \end{split}
\end{align}
and the map $d_r: E_r^{p,q}\longrightarrow E_r^{p+r,q-r+1}$ is given by
$$d_r[\alpha_{p,q}]=[\p\alpha_{p+r-1,q-r+1}],$$
where $[\alpha_{p,q}]\in E_r^{p,q}$ and $\alpha_{p+r-1,q-r+1}$ appears in \eqref{zr}.  Hence, a direct and exact application of Theorem \ref{thm1} implies
$$d_i=0,\ \forall i\geq 1,$$
which is indeed the desired degeneration.
\end{proof}
\begin{rem}
For the case that $X$ is projective and $D$ is simple normal crossing divisor, the above theorem is proved in \cite[Theorem 3.2 (b)]{Viehweg}.	
\end{rem}

\subsection{Logarithmic deformation}

In this subsection, we will discuss the logarithmic deformation by using an iterative method originally from \cite{T87, To89, LSY} and developed in \cite{Liu, RZ, RZ2, RZ15, RwZ, RwZ1, Wan, Liu1}. 

For the definition of logarithmic definition, one can refer to \cite[Definition 3]{kawa}. Let $T_X(-\log D)$ be the dual sheaf of $\Omega^1_X(\log D)$. Then the set of infinitesimal logarithmic deformations is the space $H^1(X, T_X(-\log D))$. Moreover, as shown in \cite[Page 251]{kawa}, the semi-universal family \cite[Definition 5]{kawa} can be obtain from a subspace of 
$$\Gamma_{\text{real analytic}}(X, T_X(-\log D)\otimes \Lambda^{0,1}T^*X),$$ which, usually called the space of \emph{Beltrami differentials},  consists of sections  satisfying the integrability  condition:
\begin{align}\label{integrable}
\b{\p}\varphi=\frac{1}{2}[\varphi,\varphi].
\end{align}

Suppose that $D$ is a smooth divisor with $D\in |-2K_X|$, and for any $[\varphi_1]\in H^{0,1}(X, T_X(-\log D))$ and any $t$ in a small $\epsilon$-disk $\Delta_\epsilon$ of $0$ in $\mathbb{C}^{\dim_{\mathbb{C}} H^{0,1}(X, T_X(-\log D))}$, we try to construct a holomorphic family
$$\varphi:=\varphi(t)\in A^{0,1}(X, T_X(-\log D))$$ satisfying the following integrability and initial conditions:
\begin{align}\label{log 0.1}
\b{\p}\varphi=\frac{1}{2}[\varphi,\varphi],\quad \frac{\p \varphi}{\p t}(0)=\varphi_1.	
\end{align}
To solve the above equation, we need the following lemma, and we will omit its proof because it is the same as to \cite[Lemma 4.9]{Wan}. 
\begin{lemma}\label{log lemma2}
	Let $\Omega'\in A^{0,0}(X,\Omega^n_X(\log D)\otimes K_X)$ be a logarithmic $(n,0)$-form without zero points. Then
	$$
	\bullet\l\Omega': A^{0,1}(X,T_X(-\log D))\to A^{0,1}(X, \Omega^{n-1}_X(\log D)\otimes K_X)	
	$$
is an isomorphism, whose inverse we denote by
$$
\Omega'^*\l\bullet: A^{0,1}(X, \Omega^{n-1}_X(\log D)\otimes K_X)\to A^{0,1}(X, T_X(-\log D)).
$$
\end{lemma}
 By assumption, $D\in |-2K_X|$, so $\Omega^n_X(\log D)\otimes K_X\cong \mc{O}_X(D+2K_X)$ is trivial, one may take an element 
 \begin{align}
 \Omega\in H^0(X, \Omega^n_X(\log D)\otimes K_X)	
 \end{align}
  without zero points. 
\begin{prop}\label{prop3}
If there are two smooth families $$\varphi(t)\in A^{0,1}(X, T_X(-\log D))$$ and $$\Omega(t)\in A^{0,0}(X,\Omega^{n}_X(\log D)\otimes K_X)$$ satisfying the system of equations
\begin{equation}\label{log 0.2}
\begin{cases}
 (\b{\p}+\frac{1}{2}\p\circ i_{\varphi})(i_{\varphi}\Omega(t))=0,\\
(\b{\p}+\p\circ i_{\varphi})\Omega(t)=0,\\
\Omega_0=\Omega,
\end{cases}
\end{equation}
then $\varphi(t)$ satisfies (\ref{integrable})  for sufficiently small $t$.
\end{prop}
\begin{proof}
From (\ref{log 0.2}), one has
\begin{align}\label{log 0.3}
\begin{split}
\b{\p}(\varphi\l\Omega(t))&=-\frac{1}{2}\p\circ i_{\varphi}\circ i_{\varphi}\Omega(t)\\
&=\frac{1}{2}[\varphi,\varphi]\l\Omega(t)-i_{\varphi}\circ\p\circ i_{\varphi}\Omega(t)\\
&=\frac{1}{2}[\varphi,\varphi]\l\Omega(t)+i_{\varphi}\circ\b{\p}\Omega(t).
	\end{split}
\end{align}
Therefore,
$$
(\b{\p}\varphi)\l\Omega(t)=\b{\p}(\varphi\l\Omega(t))-i_{\varphi}\circ\b{\p}\Omega(t)=\frac{1}{2}[\varphi,\varphi]\l\Omega(t).
$$
Since $\Omega(t)$ is smooth and $\Omega(0)=\Omega_0=\Omega$, $\Omega(t)\in A^{0,0}(X, \Omega^n_X(\log D)\otimes K_X)$ also has no zero point for small $t$. One has 
$$
\b{\p}\varphi=\frac{1}{2}[\varphi,\varphi].	
$$
\end{proof}

By the above Proposition, our goal is to construct two smooth families $\varphi(t)$ and $\Omega(t)$ satisfy (\ref{log 0.2}) and $(\p\varphi(t)/\p t)|_{t=0}=\varphi_1$. By taking $L=K_X^{-1}$ and using Theorem \ref{thm1}, and the same argument as \cite[Page 34-38]{Wan}, one  can solve the system of equations (\ref{log 0.2}) and  obtain
\begin{thm}\label{thm4}
Let $X$ be a compact K\"ahler manifold and $D$ a smooth divisor such that 
$D\in |-2K_X|$. Then, the pair $(X,D)$ has unobstructed deformations. More precisely, for any  $[\varphi_1]\in H^{0,1}(X, T_X(-\log D))$, there is a holomorphic family $$\varphi(t)\in A^{0,1}(X, T_X(-\log D)),$$  satisfying (\ref{log 0.1}). 
\end{thm}
\begin{rem}
For the case that $X$ is projective and $D$ is a smooth divisor with $D\in |-NK_X|$ for some positive integer $N$, $(X, D)$ was also proved to has unobstructed deformations \cite{Iacono}. More general, if $X$ is projective, $D=\sum_{i=1}^r D_i$ is a simple normal crossing divisor, and there is a collection of weights $\{a_i\}_{i\in I}\subset [0,1]\cap\mb{Q}$ so that 
$$\sum_{i\in I}a_i[D_i]=-K_X\in \text{Pic}(X)\otimes \mb{Q},$$
then the pair $(X,D)$ has also unobstructed deformations \cite[Section 4.3.3 (iii)]{KKP08}. 
\end{rem}
\section{Further discussions}\label{sec5}

Let $X$ be a smooth projective variety. The injective theorems, vanishing theorems, and unobstructed deformations have been studied widely by using the $E_1$-degeneration of some spectral sequences \cite{Viehweg, Fujino, Iacono}. In particular, by \cite[Theorem 3.2 (b)]{Viehweg}, the following spectral sequence 
\begin{align}\label{5.1}
E_1^{ab}=H^b(X,\Omega^a_X(\log D)\otimes L^{-1})\Longrightarrow	\mb{H}^{a+b}(X, \Omega^{\bullet}_X(\log D)\otimes L^{-1})
\end{align}
associated to the logarithmic de Rham complex 
$$(\Omega^{\bullet}_X(\log D)\otimes L^{-1}, \n^{(i)}_{\bullet})$$
degenerates in $E_1$. Here $L$ is a holomorphic line bundle over $X$ satisfying $L^N=\mc{O}_X(\sum_{i=1}^r a_i D_i)$, $0<a_j<N$. For the proof of (\ref{5.1}), one can first take a cyclic cover obtained by taking the $N$-th root out of a canonical section of $L^N$ and get a normal variety $Y$,  then by using a Kawamata's result \cite[Lemma 3.19]{Viehweg} (which needs an ample invertible sheaf), one can get a projective manifold $T$ and a finite morphism $\delta:T\to Y$, so that the degeneration (\ref{5.1}) can be reduced to the more familiar degeneration of the Hodge spectral sequence 
$$E_1^{ab}=H^b(T,\Omega^a_T)\Longrightarrow \mb{H}^{a+b}(T,\Omega^{\bullet}_T). $$ 
On the other hand, if $X$ is a smooth projective complex variety, $D=\sum_{i=1}^r D_i$ is a simple normal crossing divisor, and there is a collection of weights $\{a_i\}_{i\in I}\subset [0,1]\cap\mb{Q}$ so that 
\begin{align}\label{5.2}
	\sum_{i\in I}a_i[D_i]=-K_X\in \text{Pic}(X)\otimes \mb{Q},
\end{align}
then the pair $(X,D)$ has also unobstructed  deformations by using a purely algebraic method \cite[Section 4.3.3 (iii)]{KKP08}.  More precisely, they used Dolbeault type complexes to construct a differential Batalin-Vilkovisky algebra such that the associated differential graded Lie algebra (DGLA) controls the deformation problem. If the differential Batalin-Vilkovisky algebra has a degeneration property then the associated DGLA is homotopy abelian.

In terms of K\"ahler geometry, K. Liu, S. Rao and the author \cite{Wan} developed a method so that the $E_1$-degeneration of spectral sequences and the unobstructed deformations can be reduced to a logarithmic $\b{\p}$-equation. From the discussions in Section \ref{sec3}, if one can solve (\ref{dbar3}) in the case of $L^N=\mc{O}_X(\sum_{i=1}^ra_i D_i)$, $0\leq a_i\leq N$, then by the same argument as in Section 3, one can immediately obtain the degeneration (\ref{5.1}) and  give an analytic proof for the unobstructed  deformations in the case of (\ref{5.2}). Inspired by the above discussions, one may naturally  conjecture that:
\begin{con}\label{con5.1}
	Let $X$ be a compact K\"ahler manifold and $D=\sum_{i=1}^rD_i$ a simple normal crossing divisor in $X$. If $L$ is a holomorphic bundle over $X$ with $L^N=\mc{O}_X(\sum_{i=1}^r a_i D_i)$, $0\leq a_i\leq N$ and $a_i\in \mb{Z}$, then the following logarithmic $\b{\p}$-equation:
\begin{align}\label{dbar4}\b{\p}x=\p\alpha\end{align}
has a solution $x\in A^{0,q-1}(X,\Omega^{p+1}_X(\log D)\otimes L^{-1})$ for any $\alpha\in A^{0,q}(X,\Omega^p_X(\log D)\otimes L^{-1})$ with $\b{\p}\p\alpha=0$.
\end{con}

Note that (\ref{dbar4}) can be solved for the case that $L$ is trivial or $L=\mc{O}_X(D)$ \cite[Theorem 0.1, 0.2]{Wan}, and Theorem \ref{thm1} is a special case that $L^N=\mc{O}_X(D)$ for a smooth divisor $D$.


\begin{thebibliography}{99}


\bibitem{Ambro} F. Ambro, {\it An injectivity theorem}, Compositio Math.  {\bf 150} (2014), 999-1023.

\bibitem{Ancona} V. Ancona, B. Gaveau,
Differential Forms on Singular Varieties,  De Rham and Hodge Theory Simplified,
Chapman  Hall/CRC, 2005.


\bibitem{BHPV} W. Barth, K. Hulek, C. Peters, A. Van de Ven,
Compact complex surfaces,
 Second edition. Ergebnisse der Mathematik und ihrer
Grenzgebiete. 3. Folge. A Series of Modern Surveys in Mathematics
[Results in Mathematics and Related Areas. 3rd Series. A Series of
Modern Surveys in Mathematics], 4. Springer-Verlag, Berlin, 2004.


\bibitem{cfgu} L. A. Cordero, M. Fernandez, A. Gray, L. Ugarte,
 {\it A general description of the terms in the Fr\"olicher spectral sequence,}
  Diff. Geom. Applic. \textbf{7} (1997), 75-84.
  
  \bibitem{de} G. de Rham, K. Kodaira,
 \textit{Harmonic integrals,}
   (Mimeographed notes), Institute for Advanced Study, Princeton (1950).

\bibitem{De} P. Deligne,
{\it Th\'{e}or\`{e}me de Lefschetz et crit\`{e}res de d\'{e}g\'{e}n\'{e}rescence de suites spectrales,}
 Publ. Math. Inst. Hautes \'{E}tudes Sci. \textbf{35} (1969), 107-126.


\bibitem{Del70} P. Deligne,
 \textit{Equations differentielles \`{a} points singuliers r\'eguliers,}
  Springer Lect. Notes Math. \textbf{163} (1970).

\bibitem{Del71} P. Deligne, {\it Th\'eorie de Hodge, II}, Inst. Hautes \'Etudes Sci. Publ. Math. \textbf{40} (1971), 5-57.


  
  
\bibitem{Del87} P. Deligne, L. Illusie, \textit{Rel\`evements modulo $p^2$ et d\'ecomposition du complexe de de Rham}, Inventiones math. {\bf 89} (1987), 247-270.
  
  \bibitem{EV06}H. Esnault, E. Viehweg,
 \textit{Logarithmic de Rham complexes and vanishing theorems},
  Invent. Math.  \textbf{86} (1986), 161-194.

\bibitem{Viehweg} H. Esnault, E. Viehweg,
Lectures on vanishing theorems,  DMV Seminar, {\bf 20}, Birkh\"{a}user, Verlag, Basel, 1992.

\bibitem{Fujino} O. Fujino,
 \textit{Introduction to the log minimal model program for log canonical pairs,}
   preprint, arXiv:0907.1506 [math.AG]. Published as
 \textit{Foundations of the minimal model program,}
  MSJ Memoirs, 35. Mathematical Society of Japan, Tokyo, 2017.

\bibitem{Fulton} W. Fulton, Algebraic Topology, A First Course, Springer-Verlag, New York, 1995. 


\bibitem{Griffith} P. Griffith, J. Harris,
Principles of Algebraic Geometry,
 Wiley, New York, 1978.
 
 
 \bibitem{Iacono} D. Iacono, \textit{Deformations and obstructions of pairs $(X,D)$}, International Mathematics Research Notices \textbf{19} (2015), 9660-9695.
 
 
\bibitem{Kollar} J. Koll\'ar, Singularities of pairs. Algebraic Geom.-Santa Cruz. In: Proceedings of Symposia Pure Mathematics 62, Part 1, pp. 221-287. Amer. Math. Soc., Providence, RI (1995).

\bibitem{Kodaira} K. Kodaira, \textit{The theorem of Riemann-Roch on compact analytic surfaces}, Amer. J. Math. {\bf 73} (1951), 1-46.

\bibitem{KKP08} L. Katzarkov, M. Kontsevich, T. Pantev, \textit{Hodge theoretic aspects of mirror
symmetry}, In From Hodge theory to integrability and TQFT tt*-geometry,
volume 78 of Proc. Sympos. Pure Math., pages 87-174. Amer. Math. Soc., Providence,
RI, 2008.
 
 \bibitem{kawa} Y. Kawamata, {\it On deformations of compactifible complex manifolds}, Math. Ann. {\bf 235} (1978), 247-265.
 
 
 
 \bibitem{Wan} K. Liu, S. Rao, X. Wan, {\it Geometry of logarithmic forms and deformations of complex structures}, J. Algebraic Geom. 28 (2019), no. 4, 773-815. 
 
 \bibitem{Liu} K. Liu, S. Rao, X. Yang,
\textit{Quasi-isometry and deformations of Calabi-Yau manifolds},
 Invent. Math. {\bf 199} (2015), no. 2, 423-453.
 
 \bibitem{LSY} K. Liu, X. Sun, S.-T. Yau,
\textit{Recent development on the geometry of
the Teichm\"{u}ller and moduli spaces of Riemann surfaces},
 Surveys in differential geometry. Vol. XIV.
Geometry of Riemann surfaces and their moduli spaces, 221-259,
(2009).
 
 
 \bibitem{Liu1} K. Liu, S. Zhu, {\it Global methods of solving equations on manifolds}, Surveys in Differential Geometry, 2020, 241-276.
 
 \bibitem{Laz} R. Lazarsfeld, Positivity in Algebraic Geometry, I, II, 
 A Series of Modern Surveys in Mathematics 48. Berlin: Springer, 2004.
 
 \bibitem{Nog} J. Noguchi, {\it A short analytic proof of closedness of logarithmic forms}, Kodai Math. J. {\bf 18} (1995), No. 2, 295-299.
 
 \bibitem {RwZ} S. Rao, X. Wan, Q. Zhao,
 \textit{Power series proofs for local stabilities of K\"ahler and balanced structures with mild $\partial\bar\partial$-lemma},
Nagoya Mathematical Journal, 1-50. doi:10.1017/nmj.2021.4.
 
 \bibitem{RwZ1} S. Rao, X. Wan, Q. Zhao, {\it On local stabilities of p-K\"ahler structures}, Compos. Math. 155 (2019), no. 3, 455-483.

\bibitem{RZ15}
S. Rao, Q. Zhao,
\emph{Several special complex structures and their deformation properties},
J Geom Anal 28 (2018), 2984-3047.



\bibitem{T87} G. Tian,
\newblock \textit{Smoothness of the universal deformation space
of compact Calabi-Yau manifolds and its Petersson-Weil metric},
\newblock Mathematical aspects of string theory (San Diego, Calif., 1986),
629-646, Adv. Ser. Math. Phys., 1, World Sci. Publishing, Singapore,
(1987).

\bibitem {To89} A. Todorov,
\newblock \textit{The Weil-Petersson geometry of the moduli
space of $\mathbb{SU}${$(n\geq3)$} (Calabi-Yau) manifolds I},
\newblock Comm. Math. Phys., \textbf{126 (2)}, (1989), 325-346.

\bibitem{Wells} R. Wells, {\it Comparison of De Rham and Dolbeault cohomology for proper surjective mappings}, Pacific. J. Math. {\bf 53} (1974), no. 1, 281-300.

\bibitem {RZ}  Q. Zhao, S. Rao,
 \textit{Applications of deformation formula of holomorphic
one-forms},
 Pacific J. Math. Vol. {\bf 266}, No. 1, 2013,
221-255.

\bibitem {RZ2}  Q. Zhao, S. Rao,
 \textit{Extension formulas and deformation invariance of Hodge numbers},
C. R. Math. Acad. Sci. Paris {\bf 353}
(2015), no. 11, 979-984.

\end{thebibliography}
\end{document}